\documentclass{proc-l-hijacked}
\usepackage{amssymb, amsmath, amsthm, hyperref,}

\newtheorem{theorem}{Theorem}[section]

\newtheorem{proposition}[theorem]{Proposition}
\newtheorem{corollary}[theorem]{Corollary}

\theoremstyle{definition}

\newtheorem{example}[theorem]{Example}

\numberwithin{equation}{section}

\DeclareMathOperator{\Soc}{Soc}

\DeclareMathOperator{\cl}{cl}

\begin{document}

\title[Identities, Approximate Identities and TDZ]{ Identities, Approximate Identities and Topological Divisors of Zero in Banach Algebras  }
\author{F. Schulz \and R. Brits \and M. Hasse}
\address{Department of Pure and Applied Mathematics, University of Johannesburg, South Africa}
\email{francoiss@uj.ac.za, rbrits@uj.ac.za, Melanie.Hasse@standardbank.co.za}

\begin{abstract}
In \cite{tdz} S. J. Bhatt and H. V. Dedania exposed certain classes of Banach algebras in which every element is a topological divisor of zero. We identify a new (large) class of Banach algebras with the aforementioned property, namely, the class of non-unital Banach algebras which admits either an approximate identity or approximate units. This also leads to improvements of results by R. J. Loy and J. Wichmann, respectively. If we observe that every single example that appears in \cite{tdz} belongs to the class identified in the current paper, and, moreover, that many of them are classical examples of Banach algebras with this property, then it is tempting to conjecture that the classes exposed in \cite{tdz} must be contained in the class that we have identified here. However, we show somewhat elusive counterexamples. Furthermore, we investigate the role completeness plays in the results and show, by giving a suitable example, that the assumptions are not superfluous. The ideas considered here also yields a pleasing characterization: The socle of a semisimple Banach algebra is infinite-dimensional if and only if every socle-element is a topological divisor of zero in the socle.     
\end{abstract}
	\parindent 0mm
	
	\maketitle

\section{Identities, Approximate Identities and TDZ}

An element $y$ in a normed algebra $\left(A, \|\cdot\|\right)$ is called a \emph{topological divisor of zero} (or \emph{TDZ}) if there exists a sequence $\left(x_{n}\right) \subseteq A$ such that $\left\|x_{n}\right\| = 1$ for all $n \in \mathbb{N}$ and either $yx_{n} \rightarrow 0$ or $x_{n}y \rightarrow 0$. Furthermore, if $yx_{n} \rightarrow 0$ then $y$ is called a \textit{left TDZ}, and similarly, if $x_{n}y \rightarrow 0$ then $y$ is a \textit{right TDZ}. If $y$ is a left and right TDZ (where the sequences need not coincide), then $y$ is a \textit{two-sided TDZ}. The collection of all topological divisors of zero in a normed algebra $A$ will be denoted by $Z(A)$.\\

In \cite{tdz} S. J. Bhatt and H. V. Dedania established the following result concerning complex Banach algebras in which every element is a TDZ:

\begin{theorem}\cite[Theorem 1]{tdz}\label{1.1}
	Every element of a complex Banach algebra $A$ is a TDZ, if at least one of the following holds:
	\begin{itemize}
		\item[\textnormal{(i)}]
		$A$ is infinite dimensional and admits an orthogonal basis.
		\item[\textnormal{(ii)}]
		$A$ is a non-unital uniform Banach algebra in which the Shilov boundary $\partial A$ coincides with the carrier space $\mathcal{M}(A)$.
		\item[\textnormal{(iii)}]
		$A$ is a non-unital hermitian Banach $^{\ast}$-algebra with continuous involution. 
	\end{itemize}
\end{theorem}

We now expose another class of Banach algebras in which every element is a TDZ. Firstly, however, we need to recall the following definitions:\\

A net $(e_\lambda)_{\lambda\in\Lambda}$ in a normed space $A$ is called a \emph{left approximate identity} in $A$ if $\lim e_{\lambda}x=x$ for all $x\in A$. Similarly $(e_\lambda)_{\lambda\in\Lambda}$ is a \emph{right} or \emph{two-sided approximate identity} in $A$ if $\lim xe_{\lambda}=x$ for all $x\in A$ or $\lim e_{\lambda}x=x$ and $\lim xe_{\lambda}=x$ for all $x \in A$, respectively. An approximate identity $(e_\lambda)_{\lambda\in\Lambda}$ is said to be \emph{bounded} if there exists a positive constant $K$ such that $\|e_{\lambda}\|\leq K$ for each $\lambda\in\Lambda$.

\begin{theorem}\label{1.2}
	Let $A$ be a Banach algebra without a unit (resp. left unit, resp. right unit). Assume that $A$ has a two-sided (resp. left, resp. right) approximate identity $\left(e_{\lambda}\right)_{\lambda \in \Lambda}$ (which is not necessarily bounded). Then every element of $A$ is a two-sided (resp. right, resp. left) TDZ.
\end{theorem}

\begin{proof}
	We shall restrict ourselves to proving the two-sided case, as the other cases are obtained similarly. Firstly note that $\left(e_{\lambda}\right)_{\lambda \in \Lambda}$ does not converge, for then $A$ will have a unit which contradicts our hypothesis on $A$. Consequently, there exists a fixed $\epsilon > 0$ such that the following holds true: For each $\lambda_{0} \in \Lambda$, there exist $\lambda_{1}, \lambda_{2} \in \Lambda$ such that $\lambda_{1} \geq \lambda_{0}$, $\lambda_{2} \geq \lambda_{0}$ and $\left\|e_{\lambda_{1}} - e_{\lambda_{2}}\right\| \geq \epsilon$. If this was not the case then we obtain the contradiction $e_{\lambda} \rightarrow e_{\lambda_{0}}$ for some $\lambda_{0} \in \Lambda$. Let $x \in A$ be arbitrary and denote by $B(x, \delta)$ the open ball centered at $x$ with radius $\delta > 0$. Since $\left(xe_{\lambda}\right)_{\lambda \in \Lambda}$ converges to $x$, for each $n \in \mathbb{N}$, we can find a $\lambda_{n} \in \Lambda$ such that $xe_{\lambda} \in B\left(x, 1/n\right)$ whenever $\lambda \geq \lambda_{n}$. Moreover, by the preceding paragraph, we can find $\alpha_{n}, \beta_{n} \in \Lambda$ such that $\alpha_{n} \geq \lambda_{n}$, $\beta_{n} \geq \lambda_{n}$ and $\left\|e_{\alpha_{n}}-e_{\beta_{n}}\right\| \geq \epsilon$. For each $n \in \mathbb{N}$, let $y_{n} := e_{\alpha_{n}}-e_{\beta_{n}}$. Then $xy_{n} \rightarrow 0$ as $n \rightarrow \infty$ and $\left(y_{n}\right)$ does not converge to $0$. This is sufficient to conclude that $x$ is a left TDZ. Similarly, it can be shown that $x$ is a right TDZ. This completes the proof.
\end{proof}

\begin{corollary}\label{jointtdz}
	Let $A$ be a Banach algebra without a unit (resp. left unit, resp. right unit). Assume that $A$ has a two-sided (resp. left, resp. right) approximate identity $\left(e_{\lambda}\right)_{\lambda \in \Lambda}$ (which is not necessarily bounded). Then there exists a net $\left(y_{\mu}\right)_{\mu \in \Lambda_{0}}$ in $A$ such that: 
	\begin{itemize}
		\item[\textnormal{(i)}]
		$y_{\mu}x \rightarrow 0$ and $xy_{\mu} \rightarrow 0$ for every $x \in A$ (resp. $y_{\mu}x \rightarrow 0$, resp. $xy_{\mu} \rightarrow 0$).
		\item[\textnormal{(ii)}]
		$y_{\mu} \not\rightarrow 0$.
	\end{itemize}
\end{corollary}

\begin{proof}
	Suppose that the non-unital Banach algebra $A$ has a two-sided approximate identity $\left(e_{\lambda}\right)_{\lambda \in \Lambda}$ (the other cases follow similarly). Recall from the proof of Theorem \ref{1.2} that there exists an $\epsilon > 0$ such that for each $\lambda_{0} \in \Lambda$, there exist $\lambda_{1}, \lambda_{2} \in \Lambda$ such that $\lambda_{1} \geq \lambda_{0}$, $\lambda_{2} \geq \lambda_{0}$ and $\left\|e_{\lambda_{1}}-e_{\lambda_{2}}\right\| \geq \epsilon$. We now define
	$$\Lambda_{0} := \left\{\left(\alpha, \gamma\right) \in \Lambda \times \Lambda: \left\|e_{\alpha}-e_{\gamma}\right\| \geq \epsilon \right\}$$
	and for $\left(\alpha_{1}, \gamma_{1}\right), \left(\alpha_{2}, \gamma_{2}\right) \in \Lambda_{0}$,  $\left(\alpha_{1}, \gamma_{1}\right) \leq \left(\alpha_{2}, \gamma_{2}\right)$ if and only if $\alpha_{1} \leq \alpha_{2}$ and $\gamma_{1} \leq \gamma_{2}$ in $\Lambda$. We claim that $\left(\Lambda_{0}, \leq \right)$ is a directed set: Reflexivity and transitivity follows readily from the fact that $\Lambda$ is directed. So let $\left(\alpha_{1}, \gamma_{1}\right), \left(\alpha_{2}, \gamma_{2}\right) \in \Lambda_{0}$. Since $\Lambda$ is directed, there exists a $\lambda_{0}$ such that $\lambda_{0} \geq \alpha_{j}$ and $\lambda_{0} \geq \gamma_{j}$ for $j =1, 2$. Now, by hypothesis, there exists $\beta_{1}, \beta_{2} \in \Lambda$ such that $\beta_{1} \geq \lambda_{0}$, $\beta_{2} \geq \lambda_{0}$ and $\left\|e_{\beta_{1}}-e_{\beta_{2}}\right\| \geq \epsilon$. Hence, $\left(\beta_{1}, \beta_{2}\right) \in \Lambda_{0}$ and $\left(\alpha_{j}, \gamma_{j}\right) \leq \left(\beta_{1}, \beta_{2}\right)$ for $j = 1, 2$. This proves our claim. For each $\mu := \left(\alpha, \gamma\right) \in \Lambda_{0}$, let $y_{\mu} := e_{\alpha} - e_{\gamma}$. We now show that $\left(y_{\mu}\right)_{\mu \in \Lambda_{0}}$ is the desired net: Certainly (ii) is satisfied since $\left\|y_{\mu}\right\| \geq \epsilon$ for all $\mu \in \Lambda_{0}$. It therefore remains to verify (i). Let $x \in A$ be arbitrary, and recall that $\left(e_{\lambda}x\right)_{\lambda \in \Lambda}$ converges to $x$. Fix any $\delta > 0$ and consider the open ball $B\left(0, \delta\right)$. Since $e_{\lambda}x \rightarrow x$, there exists a $\lambda_{0} \in \Lambda$ such that $e_{\lambda}x \in B\left(x, \delta/2\right)$ whenever $\lambda \geq \lambda_{0}$. Moreover, there exists $\lambda_{1}, \lambda_{2} \in \Lambda$ such that $\lambda_{1} \geq \lambda_{0}$, $\lambda_{2} \geq \lambda_{0}$ and $\left\|e_{\lambda_{1}}-e_{\lambda_{2}}\right\| \geq \epsilon$. Thus, $\left(\lambda_{1}, \lambda_{2}\right) \in \Lambda_{0}$. Now, if $\left(\alpha, \gamma\right) \in \Lambda_{0}$ and $\left(\alpha, \gamma\right) \geq  \left(\lambda_{1}, \lambda_{2}\right)$, then $\alpha \geq \lambda_{1} \geq \lambda_{0}$ and $\gamma \geq \lambda_{2} \geq \lambda_{0}$. Consequently,
	\begin{eqnarray*}
		\left\|\left(e_{\alpha}-e_{\gamma}\right)x\right\| & = & \left\|e_{\alpha}x -x + x -e_{\gamma}x\right\| \\
		& \leq & \left\|e_{\alpha}x - x\right\| + \left\| e_{\gamma}x-x\right\| \\
		& < & \frac{\delta}{2} + \frac{\delta}{2} \;\, =\;\, \delta.
	\end{eqnarray*}
	This shows that $\left(\alpha, \gamma\right) \geq  \left(\lambda_{1}, \lambda_{2}\right)$ implies $\left(e_{\alpha}-e_{\gamma}\right)x \in B\left(0, \delta\right)$. Hence, we conclude that the net $\left(y_{\mu}x\right)_{\mu \in \Lambda_{0}}$ converges to $0$. Similarly, one can prove that $\left(xy_{\mu}\right)_{\mu \in \Lambda_{0}}$ converges to $0$ since $xe_{\lambda} \rightarrow x$. The result now follows.
\end{proof}

\textbf{Remark.} If a net in $A$ satisfies both properties (i) and (ii) in Corollary \ref{jointtdz}, then $A$ is said to consist entirely of joint topological divisors of zero (see \cite[p. 88]{nonremovableideals}).\\

Completeness plays a central role in the proof of Theorem \ref{1.2}. The next example emphasizes this. It exhibits a non-unital normed algebra which has a bounded two-sided approximate identity, but which has at least one element which is not a TDZ:

\begin{example}\label{1.3}
	\textnormal{Let $B(H)$ be the Banach algebra of bounded linear operators from a infinite dimensional complex Hilbert space $H$ into itself. Denote by $A$ the $C^{\ast}$-subalgebra of $B(H)$ generated by the identity operator $I_{H}$ and $K(H)$, the ideal of compact operators on $H$. Fix any non-algebraic operator $T \in K(H)$ (for instance a compact operator with an infinite spectrum). We now define $B$ to be the subalgebra of $A$ generated by the element $T+I_{H}$. Finally, if we denote by $F(H)$ the ideal of finite rank operators in $B(H)$, then by \cite[Proposition 3]{ultraprimeexa} $C:= F(H)+B$ is a dense non-unital subalgebra of the unital Banach algebra $A$. So there exits a non-unital normed algebra whose completion is unital. Let $\left(E_{n}\right)$ be the sequence in $C$ such that $E_{n} \rightarrow I_{H}$ as $n \rightarrow \infty$. Then $\left(E_{n}\right)$ is a bounded two-sided approximate identity in $C$. However, if every element in $C$ is a TDZ, then by the denseness of $C$ every element of $A$ is a TDZ. But this is absurd since $I_{H} \in A$. Hence, $C$ contains at least one element which is not a TDZ.}
\end{example}

The closure of a set $Y$ in a topological space $X$ is denoted by $\cl(Y)$. If $(A, \|\cdot\|)$ is a normed space then we shall also write $\cl(A)$ for the completion of $A$ under the norm $\|\cdot\|$. It is possible to deduce the following from Theorem \ref{1.2}:

\begin{corollary}\label{1.4}
	Let $A$ be a normed algebra and assume that $\cl(A)$ does not have a unit (resp. left unit, resp. right unit). If $A$ has a two-sided (resp. left, resp. right) approximate identity $\left(e_{\lambda}\right)_{\lambda \in \Lambda}$ (which is not necessarily bounded), then every element of $A$ is a two-sided (resp. right, resp. left) TDZ (in $A$).
\end{corollary}

\begin{proof}
	If $\left(e_{\lambda}\right)_{\lambda \in \Lambda}$ converges in $\cl(A)$, then we obtain a contradiction with the hypothesis on $\cl(A)$. The argument in the proof of Theorem \ref{1.2} now readily establishes the result. 
\end{proof}

A similar result holds true for algebras with approximate units. A normed algebra $A$ is said to have \emph{left approximate units} if for every $x\in A$ and every $\epsilon>0$ there exists a $u\in A$ (depending on $x$ and $\epsilon$) such that  $\|x-ux\|<\epsilon$. Similarly, in the obvious natural way, we can define the notions of \emph{right} and \emph{two-sided approximate units} in $A$. Furthermore we note that $A$ is said to have, for instance, \emph{bounded left approximate units} if there exists a positive constant $K$ such that for every $x\in A$ and every $\epsilon>0$ there exists a $u\in A$ (depending on $x$ and $\epsilon$) such that 
$$\|u\|\leq K\;\,\mathrm{and}\;\, \|x-ux\|<\epsilon.$$
Finally we mention that $A$ has, for instance, \emph{pointwise-bounded left approximate units} if for every $x\in A$ and every $\epsilon>0$ there exists a $u\in A$ (depending on $x$ and $\epsilon$) and a positive constant $K(x)$ such that
$$\|u\|\leq K(x)\;\,\mathrm{and}\;\, \|x-ux\|<\epsilon.$$

\begin{theorem}\label{appunitstdz}
	Let $A$ be a normed algebra such that $\cl(A)$ is not unital. If $A$ has left (or right or two-sided) approximate units then every element of $\cl(A)$ is a TDZ. In particular, every element of a non-unital Banach algebra with left approximate units is a TDZ. 
\end{theorem}	
\begin{proof}
	Suppose first $a\in A$ is not a TDZ in $\cl(A)$. Since $A$ has left approximate units there exists a sequence $(u_n)$ in $A$ such that $\lim u_na=a$. If  $\left(u_n\right)$ is not Cauchy then we can find $\epsilon>0$, and  two subsequences of $(u_n)$, say $(u_{n_k})$ and $(u_{m_k})$, such that $\|u_{n_k}-u_{m_k}\|\geq\epsilon$ for all $k\in\mathbb N$. But then
	$$\lim_k \frac{(u_{n_k}-u_{m_k})a}{\|u_{n_k}-u_{m_k}\|}=0$$ shows that $a$ is a TDZ which contradicts the hypothesis. Thus $\left(u_n\right)$ must be Cauchy, with limit say $u\in \cl(A)$. From this it follows that $(u^2-u)a=0$ and hence, if $u$ is not an idempotent, $a$ is a divisor of zero in $\cl(A)$ and thus a TDZ which contradicts the assumption on $a$. So $u$ is an idempotent satisfying $ua=a$. Necessarily $u$ commutes with $a$ because otherwise, if $au-a\neq 0$, we get that
	$(au-a)a=0$ so that $a$ is a divisor of zero in $\cl(A)$ which again contradicts the assumption on $a$. Thus $ua=au=a$. But, by assumption, there must exist some $x\in\cl(A)$ such that either $ux-x\neq 0$ or $xu-x\neq 0$. If the first instance occurs we have
	$a(ux-x)=0$; and if the second case holds $(xu-x)a=0$. Again this gives a contradiction. So $\cl(A)$ contains a dense set of topological divisors of zero, and it follows that each element of $\cl(A)$ is a TDZ.    
\end{proof}

\textbf{Remark.} Example \ref{1.3} also shows that the assumption that $\cl(A)$ is not unital in Theorem \ref{appunitstdz} is not superfluous.\\

In \cite[Proposition 3]{idintensorproducts} R. J. Loy proves that if a Banach algebra $A$ does not consist entirely of right (left) topological divisors of zero and has a left (right) approximate identity, then it has a bounded left (right) approximate identity. In light of Theorem \ref{1.2} much more is true:

\begin{corollary}
	Let $A$ be a Banach algebra which does not consist entirely of right (resp. left) topological divisors of zero. If $A$ has a left (resp. right) approximate identity, then $A$ has a left (resp. right) unit. In particular, if moreover $A$ is commutative, then $A$ is unital.
\end{corollary}

Loy remarks further, in his paper, that the converse of \cite[Proposition 3]{idintensorproducts} is not true. In particular he gives an example of a Banach algebra and then observes the following: ``..., so that all the elements are topological divisors of zero, \textit{but} has a (countable) bounded approximate identity''. We now know that every element is a TDZ (in his example) \textit{because} it has an approximate identity. In a similar vein J. Wichmann states and proves in \cite[Theorem 2]{boundappunits} that a commutative normed algebra $A$ which does not consist entirely of topological divisors of zero has pointwise-bounded approximate units if and only if $A$ has a bounded approximate identity. So suppose $A$ has at least one element, say $a$, which is not a TDZ in $A$. Then $a$ is not a TDZ in the completion of $A$ either. Now if $A$ has a bounded approximate identity $\left(e_{\lambda}\right)$ in $A$, then it is easy to see that $\left(e_{\lambda}\right)$ is also a bounded approximate identity for $\cl(A)$. So again, by Theorem \ref{1.2}, every element of $\cl(A)$ must be a TDZ. But this contradicts the fact that $a$ is not a TDZ in $\cl(A)$.  Hence, under the conditions above, it must in fact be the case that $\cl(A)$ is unital. This then is an improvement of Wichmann's result:

\begin{corollary}
	Let $A$ be a commutative normed algebra which does not consist entirely of topological divisors of zero. If $A$ has approximate units or an approximate identity, then $\cl(A)$ is unital.
\end{corollary}   

It is easy to see that the class of Banach algebras identified in Theorem \ref{1.2} contains those Banach algebras mentioned in (i) of Theorem \ref{1.1}. Indeed, let $\left(e_{n}\right)$ be the orthogonal basis of $A$. For each positive integer $n$, set $k_{n} := \sum_{i=1}^{n} e_{i}$. Then, since each $x \in A$ can be expressed as $x = \sum_{m =1}^{\infty} \alpha_{m}e_{m}$, where the $\alpha_{m}$'s are scalars, and since $e_{m}e_{n} = \delta_{mn}e_{n}$, $\delta_{mn}$ being the Kronecker delta, it readily follows that
$$\lim_{n \rightarrow \infty} xk_{n} = \lim_{n \rightarrow \infty} k_{n}x = x$$
for all $x \in A$. So, in this case, $A$ has a two-sided approximate identity. In particular, S. J. Bhatt and V. Dedania pointed out in \cite[Example 3.1]{tdz} that the following algebras have orthogonal bases:
\begin{itemize}
	\item[(i)]
	For the unit circle $T$, the Banach convolution algebra (Lebesgue space) $L^{p}(T)$, $1 < p < \infty$.
	\item[(ii)]
	The Banach sequence algebras $c_{0}$, $\ell^{p}$  ($1\leq p < \infty$), with pointwise multiplication.
	\item[(iii)]
	The Hardy spaces $H^{p}(U)$ ($1 < p < \infty$) on the open unit disk $U$. 
\end{itemize}   
However, it is significantly more difficult to decide about the containment for the latter two classes of Banach algebras in Theorem \ref{1.1}. To emphasize this, we revisit some further examples which appear in \cite{tdz}:
\begin{itemize}
	\item[(1)]
	For a locally compact nondiscrete abelian group $G$, the convolution algebra $L^{1}(G)$ is an example of a non-unital hermitian Banach $^{\ast}$-algebra with continuous involution. Moreover, $L^{1}(G)$ has an approximate identity (see for instance \cite[p. 321]{rickart1960general}).
	\item[(2)]
	The subalgebras $C(T)$ (continuous functions) and $C^{m}(T)$ ($C^{m}$-functions) of $L^{1}(T)$ with respective norms
	$$\left\|f\right\|_{\infty} = \sup_{t \in T} \left|f(t)\right| \;\,\mathrm{and}\;\,\left\|f\right\|_{m} = \sup_{t \in T} \sum_{j=0}^{m} \frac{\left|f^{(j)}(t)\right|}{j!},$$
	are examples of non-unital hermitian Banach $^{\ast}$-algebras with continuous involution. Moreover, these algebras are homogeneous on $T$. Consequently, it follows from \cite[Theorem 2.11]{katznelson2004introduction} that Fej\'{e}r's kernel is an approximate identity in both algebras.
\end{itemize}
The authors of \cite{tdz} did not explicitly give an example of a Banach algebra satisfying condition (ii) in Theorem \ref{1.1}. However, by the Gelfand-Naimark Theorem it can straightforwardly be established that every non-unital commutative $C^{\ast}$-algebra $A$ satisfies $\partial A = \mathcal{M}(A)$. Moreover, it is well known that every non-unital commutative $C^{\ast}$-algebra contains bounded approximate units and hence a bounded approximate identity (see \cite[Lemma 2.10.1]{arveson2002short} and \cite[Proposition 2.9.14(ii)]{dales2000banach}, respectively).\\

In spite of the above, it turns out that each of the classes (ii) and (iii) in Theorem \ref{1.1} contains a Banach algebra without an approximate identity. We discuss these examples below:

\begin{example}\label{1.5}
	\textnormal{Let $D: = \left\{z \in \mathbb{C}: \left|z\right|<1 \right\}$, $\cl(D)$ be the closure of $D$ in $\mathbb{C}$ and let $I:= \left[0, 1\right]$. Take $B$ to be the so-called ``tomato can algebra''; that is, $B$ is the uniform algebra of all continuous complex-valued functions $f$ on $K:= \cl(D)\times I$ such that the function $z \mapsto f\left(z, 1\right)$ from $\cl(D)$ into $\mathbb{C}$ is analytic on $D$ (and continuous on $\cl(D)$). H. G. Dales and A. \"{U}lger have observed in \cite[Example 4.8(ii)]{approxident} that $B$ is natural, that is, $K$ can be identified with $\mathcal{M}(B)$ via the mapping $x \mapsto \chi_{x}$, where $\chi_{x}$ is the evaluation functional at $x$. Moreover, they showed that the closed ideal
		$$A := \left\{f \in B: f\left(0, 1\right) = 0 \right\}$$
		of $B$ does not have an approximate identity. We now prove that $\partial A = \mathcal{M}(A)$. Recall that the hull of $A$ viewed as an ideal of $B$ is given by
		$$\mathcal{H}(A) := \left\{\chi \in \mathcal{M}(B): \chi\left(A\right) = \left\{0\right\} \right\}.$$
		Certainly, $A$ contains the function 
		$$\left(z, \alpha\right) \mapsto z \;\,\left(\left(z, \alpha\right) \in K\right).$$
		Moreover, since $K$ is a compact Hausdorff space, it follows that $K$ is normal. Hence, by Urysohn's Lemma, for each $x \in K - \left(\cl(D) \times \left\{1\right\}\right)$, there exists an $f \in A$ such that $f\left(\cl(D) \times \left\{1\right\}\right) = \left\{0 \right\}$ and $f(x) = 1$. Consequently, since all the characters in $\mathcal{M}(B)$ are evaluation functionals, it readily follows that $\mathcal{H}(A)= \left\{\chi_{(0, 1)} \right\}$. Now, by \cite[Proposition 4.1.11]{dales2000banach} we may infer that the mapping $\chi \mapsto \left.\chi\right|_{A}$ from $\mathcal{M}(B)-\mathcal{H}(A)$ into $\mathcal{M}(A)$ is a homeomorphism. So $\mathcal{M}(A)$ consists of evaluation functionals at the points in $K - \left\{\left(0, 1\right)\right\}$. By the definition of a weak$^{\ast}$-open set in $\mathcal{M}(A)$, it is easy to see that if $x_{n} \rightarrow x$ as $n \rightarrow \infty$ in $K - \left\{\left(0, 1\right)\right\}$, then  $\chi_{x_{n}} \rightarrow \chi_{x}$ as $n \rightarrow \infty$ in $\mathcal{M}(A)$. Thus, in order to prove that $\partial A = \mathcal{M}(A)$, it will suffice to show that $\chi_{y} \in \partial A$ for each $y \in K - \left(\cl(D) \times \left\{1\right\}\right)$ (since $\partial A$ is closed in $\mathcal{M}(A)$). Let $y \in K - \left(\cl(D) \times \left\{1\right\}\right)$ be arbitrary. As above we choose $f$ such that $f\left(\cl(D) \times \left\{1\right\}\right) = \left\{0 \right\}$ and $f(y) = 1$. Observe that $K$ is metrizable and denote by $d$ its metric. Next we consider the continuous function $g: K \rightarrow \mathbb{C}$ defined by
		$$g(x) = \frac{1}{1+d(x, y)} \;\,\left(x \in K\right).$$
		If we let $h(x) = f(x)g(x)$ for each $x \in K$, then $h \in A$. Moreover, $\left|h(x)\right| < 1$ whenever $x \in K-\left\{y \right\}$ and $h(y) = 1$. Hence, $\left|\chi (h)\right|< 1$ for all $\chi \in \mathcal{M}(A)-\left\{\chi_{y} \right\}$ and $\left|\chi_{y} (h)\right| = 1$. It therefore follows that $\chi_{y} \in \partial A$ for each $y \in K - \left(\cl(D) \times \left\{1\right\}\right)$, and so, $\partial A = \mathcal{M}(A)$ as advertised. This example shows that the hypotheses in part (ii) of Theorem \ref{1.1} need not imply that $A$ has an approximate identity.  
	}
\end{example}

\begin{example}\label{1.6}
	\textnormal{Denote by $c_{0}$ the Banach algebra of all sequences of complex numbers which converge to $0$, equipped with the norm
		$$\left\|\alpha\right\|_{\infty} = \sup_{k \in \mathbb{N}} \left|\alpha_{k}\right| \;\,\left( \alpha = \left(\alpha_{k} \right) \in c_{0}\right).$$
		For $\alpha = \left(\alpha_{k} \right) \in c_{0}$, set
		$$p_{n} \left(\alpha\right) := \frac{1}{n}\sum_{k=1}^{n} k\left|\alpha_{k+1}-\alpha_{k}\right| \;\,\left(n \in \mathbb{N}\right).$$
		Define $A:= \left\{\alpha \in c_{0}: \sup_{n \in \mathbb{N}}p_{n}\left(\alpha\right) < \infty \right\}$ and 
		$$\left\|\alpha\right\| := \left\|\alpha\right\|_{\infty} + p\left(\alpha\right)\;\,\left(\alpha \in A\right),$$
		where $p\left(\alpha\right) := \sup_{n \in \mathbb{N}}p_{n}\left(\alpha\right)$. It can then be verified that $\left(A, \left\|\cdot\right\|\right)$ is a non-unital complex Banach algebra. In fact, since termwise complex conjugation defines an involution on $A$ and $\left\|\alpha\right\| = \left\|\alpha^{\ast}\right\|$ for all $\alpha \in A$, it readily follows that $A$ is a non-unital hermitian Banach $^{\ast}$-algebra with continuous involution. This example is due to J. Feinstein and is discussed in more rigorous detail in \cite[Example 4.1.46]{dales2000banach}. In particular, it is shown there that
		$$A^{2}:= \left\{\alpha\beta : \alpha \in A, \beta \in A \right\}$$
		is separable, but that $A$ is non-separable. Hence, $A$ does not have an approximate identity. So this example shows that the hypotheses in part (iii) of Theorem \ref{1.1} need not imply that $A$ has an approximate identity.}
\end{example}

We now proceed to show that there is a large class of normed algebras each containing a two-sided approximate identity, but whose completions are non-unital. The following results will be useful in this regard:

\begin{theorem}\label{1.7}
	Let $A$ be a complex semisimple Banach algebra with a unit, and suppose that the socle of $A$, denoted $\Soc(A)$, is nonzero. Then $\Soc(A)$ has a two-sided approximate identity.
\end{theorem}

\begin{proof}
	Let $x_{1}, \ldots, x_{n} \in \Soc(A)$. By \cite[Theorem 3.13]{commutatorsinsocle} there exists a subalgebra $B$ of $\Soc(A)$ such that
	$$x_{1}, \ldots, x_{n} \in B \cong M_{n_{1}}\left(\mathbb{C}\right) \oplus \cdots \oplus M_{n_{k}}\left(\mathbb{C}\right),$$
	where the operations in the latter algebra are all pointwise. Let $e$ be the unit of $B$. Then $x_{j}e=ex_{j}=x_{j}$ for each $j \in \left\{1, \ldots, n \right\}$. By the remarks in \cite[\S 1]{unbapproxid} this is sufficient to infer the existence of a two-sided approximate identity in $\Soc(A)$. 
\end{proof}

Let $A$ be a complex semisimple Banach algebra with a unit. By \cite[Theorem 2.2]{tracesocleident} it follows that $\Soc(A)$ is finite-dimensional if and only if it is closed in $A$. Moreover, by the remark after Theorem 2.2 in \cite{tracesocleident} it follows that if $\Soc(A)$ is finite-dimensional, then $\Soc(A)$ has the Wedderburn-Artin structure; that is, $\Soc(A)$ is isomorphic as an algebra to $M_{n_{1}}\left(\mathbb{C}\right) \oplus \cdots \oplus M_{n_{k}}\left(\mathbb{C}\right)$ (where the operations in the latter algebra are all pointwise).   

\begin{proposition}\label{1.8}
	Let $A$ be a complex semisimple Banach algebra with a unit. Then $\Soc(A)$ is finite-dimensional if and only if the closure of $\Soc(A)$ has an identity element.
\end{proposition}

\begin{proof}
	By the paragraph preceding the proposition, the forward implication directly follows, and the reverse implication will be established if we can prove that if the closure of $\Soc(A)$ has an identity element, then $\Soc(A)$ is closed. To this end, denote by $B$ the closure of $\Soc(A)$ in $A$ and let $e$ be the identity element of $B$. Then, since $\Soc(A)$ is an ideal, $B = eAe$. Hence, by \cite[Lemma 2.5]{spectrumpreservers} it readily follows that $B$ is a semisimple Banach algebra with identity $e$. But $\Soc(A)$ is dense in $B$. Hence, there exists a sequence $\left(e_{n}\right) \subseteq \Soc(A)$ such that $e_{n} \rightarrow e$ as $n \rightarrow \infty$. So, since $e$ is the identity of $B$, it follows that $e_{n}$ must be invertible in $B$ for all $n$ sufficiently large. Thus, since $\Soc(A)$ is an ideal, it follows that $e \in \Soc(A)$ and consequently we have $B \subseteq \Soc(A)$. So $\Soc(A)$ is closed which establishes the result.
\end{proof}

\begin{theorem}\label{1.9}
	Let $A$ be a complex semisimple Banach algebra with a unit. Then $\Soc(A)$ is infinite-dimensional if and only if every element of $\Soc(A)$ is a TDZ in $\Soc(A)$.
\end{theorem}

\begin{proof}
	This follows immediately from Theorem \ref{1.7}, Proposition \ref{1.8} and Corollary \ref{1.4}.
\end{proof}

A recent paper \cite{approxident} of H. G. Dales and A. \"{U}lger investigates (various notions of) approximate identities in function algebras. In Theorem~\ref{uniform} we show that every non-unital commutative Banach algebra $A$ with $A\not=Z(A)$ generates a function (uniform) algebra which does not possess an approximate identity.  We first need to establish Proposition~\ref{unitization}.\\

If $A$ is any Banach algebra, then denote by $A_\mathbf 1$ the \emph{standard unitization} of $A$. This turns $A$ into a unital Banach algebra via extension.

\begin{proposition}\label{unitization}
	Let $A$ be a non-unital Banach algebra, and let $a\in A$. Then $a$ is a TDZ in $A$ if and only if $a$ is a TDZ in $A_\mathbf 1$. 
\end{proposition}

\begin{proof}
	The forward implication is obvious. For the reverse implication assume $a$ is a TDZ in $A_{\mathbf 1}$. Then we can find, without loss of generality, a sequence $(z_n)$ in $A$, and a sequence of complex numbers $(\lambda_n)$ such that 
	$$\|z_n\|+|\lambda_n|=1\mbox{ for all } n\in\mathbb N\mbox{ and, }\lim a(z_n+\lambda_n\mathbf 1)=0.$$
	So, the sequence $(\lambda_n)$ is bounded, and, by the Bolzano-Weierstrass Principle, we can assume without loss of generality that it converges. Say $\lim\lambda_n=\lambda$. Thus $\lim az_n=-\lambda a$. If $\lambda=0$, then $\lim az_n/\|z_n\|=0$, and the proof is complete. Suppose $\lambda\not=0$. Then there is a sequence, say $(y_n)$, in $A$ such that $\lim ay_n=a$. If $(y_n)$ does not converge then it is not Cauchy in $A$, and similar to the argument used in the proof of Theorem~\ref{appunitstdz} we may conclude that $a$ is a TDZ in $A$. Assume therefore that $(y_n)$ converges; if $(y_n)$ has limit, say $y\in A$, then $a(y^2-y)=0$ from which it follows (as in the proof of Theorem~\ref{appunitstdz}) that $a$ is a TDZ in $A$, unless $y$ is an idempotent of $A$ commuting with $a$. But if the latter case prevails then, since $A$ is non-unital, and $y\in A$, we can again argue as in the proof of Theorem~\ref{appunitstdz} to conclude that $a$ is a TDZ in $A$. 
\end{proof}

For a commutative Banach algebra $A$ we denote by $x\mapsto\hat x$ the Gelfand transform of $A$, and write $\hat A:=\{\hat a:a\in A\}$. $\hat A$ is a normed algebra under the spectral radius $\rho_A(\cdot)$ for elements of $A$. That is, if $\hat a\in\hat A$, then 
$\|\hat a\|=\rho_A(a)$ defines a (possibly incomplete) norm on $\hat A$. 

\begin{theorem}\label{uniform}
	Let $A$ be a non-unital commutative Banach algebra. If some $a\in A$ is not a TDZ in $A$ then $\cl(\hat A)$ does not have an approximate identity.
\end{theorem}

\begin{proof}
	If $a$ is not a TDZ in $A$, then, by Proposition~\ref{unitization}, $a$ is not a TDZ in $A_\mathbf{1}$. A result of Arens' (see for instance \cite[p. 48]{larsen1973banach}) then says that $a$ is invertible in some Banach superalgebra, say $B$, of $A_\mathbf{1}$. Thus $\hat a$ cannot be a TDZ in
	$\hat A_\mathbf{1}$. Observe that, since  $\hat A_\mathbf{1}\cong( \hat A )_\mathbf 1$, $\hat a$ is not a TDZ in $\hat A$, and thus also not a TDZ in $\cl(\hat A)$. So $\cl(\hat A)$ cannot have an approximate identity. 
\end{proof}

\bibliographystyle{amsplain}
\bibliography{Spectral}

\end{document}